\documentclass[11pt]{article}
\usepackage{amsmath, amsfonts, amsthm, amssymb, color}
\usepackage{graphicx}
\usepackage{float}
\usepackage{verbatim}

\allowdisplaybreaks

\numberwithin{equation}{section}

\newtheorem{lemma}{Lemma}[section]

\newtheorem{thm}[lemma]{Theorem}
\newtheorem{cor}[lemma]{Corollary}
\theoremstyle{definition}

\newtheorem{example}[lemma]{Example}

\theoremstyle{remark}
\newtheorem{remark}[lemma]{Remark}

\def\R{\mathbb{R}}
\def\N{\mathbb{N}}
\def\t{\mathbf{t}}

\def\D{\mathcal{D}}

\def\O{\mathcal{O}}
\def\a{\mathbf{a}}
\def\b{\mathbf{b}}

\def\p{\mathbf{p}}
\def\I{\mathcal{I}}
\def\x{\mathbf{x}}

\def\t{\mathbf{t}}
\def\v{\mathbf{v}}
\def\u{\mathbf{u}}

\numberwithin{equation}{section} \numberwithin{table}{section}

\title{Iterated function systems with super-exponentially close cylinders II}
\author{Simon Baker\\ \\
\emph{School of Mathematics,} \\ \emph{University of Birmingham,} \\ \emph{Birmingham,  B15 2TT, UK.} \\ Email: simonbaker412@gmail.com\\}

\date{\today}
\begin{document}
\maketitle

\begin{abstract}
Until recently, it was an important open problem in Fractal Geometry to determine whether there exists an iterated function system acting on $\R$ with no exact overlaps for which cylinders are super-exponentially close at all small scales. Iterated function systems satisfying these properties were shown to exist by the author and by B\'{a}r\'{a}ny and K\"{a}enm\"{a}ki. In this paper we prove a general theorem on the existence of such iterated function systems within a parameterised family. This theorem shows that if a parameterised family contains two independent subfamilies, and the set of parameters that cause exact overlaps satisfies some weak topological assumptions, then the original family will contain an iterated function system satisfying the desired properties. We include several explicit examples of parameterised families to which this theorem can be applied.
\\




\end{abstract}

\section{Introduction}
Let $\Phi=\{\phi_i:\R^d\to \R^d\}_{i\in\I}$ denote a finite set of contracting similarities acting on $\mathbb{R}^d$. We call $\Phi$ an iterated function system or IFS for short. Iterated function systems are useful tools for generating fractal sets. A well known result due to Hutchinson \cite{Hut} states that for any IFS $\Phi$ there exists a unique non-empty compact set $X\subset \R^d$ satisfying $$X=\bigcup_{i\in \I} \phi_i(X).$$ We call $X$ the self-similar set of $\Phi.$ Self-similar sets often exhibit fractal properties. The  middle-third Cantor set and the von-Koch curve are well known examples of self-similar sets. 

A well studied and difficult problem is to determine the Hausdorff dimension of a general self-similar set. Given an IFS $\Phi=\{\phi_i\}_{i\in \I}$ we denote by $r_i$ the similarity ratio of $\phi_i$. We call the unique $s\geq 0$ satisfying $\sum_{i\in \I}r_i^s=1$ the similarity dimension of $\Phi$ and denote it by $\dim_{S}\Phi$. The following upper bound for the Hausdorff dimension of a self-similar set is well known 
\begin{equation}
\label{similarity dimension inequality}
\dim_{H}(X)\leq \min\{\dim_{S}\Phi,d\}.
\end{equation} It is often the case that equality holds in \eqref{similarity dimension inequality}. To make progress with the problem of determining the Hausdorff dimension of a self-similar set one often uses self-similar measures. These are defined as follows: Given an IFS $\Phi$ and a probability vector $\p=(p_i)_{i\in \I}$, then there exists a unique Borel probability measure $\mu_{\p}$ satisfying $$\mu_{\p}=\sum_{i\in \I}p_i\cdot \mu_{\p}\circ \phi_{i}^{-1}.$$ We call $\mu_{\p}$ the self-similar measure corresponding to $\Phi$ and $\p$. We define the dimension of a Borel probability measure $\mu$ to be $$\dim \mu:=\inf\{\dim_{H}(A):\mu(A)=1\}.$$ We remark that there are other well-studied notions of dimension for measures. Importantly for self-similar measures these alternative definitions typically give the same value as our definition. This is a consequence of the exact dimensionality of self-similar measures (see \cite{FengHu}). The following upper bound for the dimension of a self-similar measure is well known
\begin{equation}
\label{similarity dimension for measures}
\dim \mu_{\p}\leq \min\left\{\frac{\sum_{i\in \I}p_i\log p_i}{\sum_{i\in \I}p_i\log r_i},d\right\}.
\end{equation} Note that if we take $\p=(r_i^{\dim_{S}\Phi})_{i\in \I}$ then equality in \eqref{similarity dimension for measures} implies equality in \eqref{similarity dimension inequality}. Self-similar measures are often easier to analyse then self-similar sets. As such, proving that equality holds in \eqref{similarity dimension for measures} often provides an easier route to proving equality in \eqref{similarity dimension inequality}.

Self-similar measures are well known examples of multifractal measures, i.e. measures which exhibit different rates of scaling on small balls. A useful tool for describing the multifractal behaviour of a measure is provided by the $L^q$ dimension. Let $q>1$ and $\mu$ be a Borel probability measure on $\mathbb{R}^d,$ we define the $L^q$ dimension of $\mu$ to be
$$D(\mu,q):=\liminf_{n\to\infty}\frac{-\log \sum_{(j_1,\ldots,j_d)\in \mathbb{Z}^d}\mu([j_1\cdot2^{-n},(j_1+1)\cdot2^{-n})\times \cdots \times[j_d\cdot2^{-n},(j_d+1)\cdot2^{-n}))^q }{n}.$$ Let $T(\mu_{\p},q)$ be the unique solution to the equation $\sum_{i\in \I}p_i^q r_i^{-s}=1.$ The following upper bound for the $L^q$ dimension of a self-similar measure always holds
\begin{equation}
\label{Lq upper bound}
D(\mu_{\p},q)\leq \min\left\{\frac{T(\mu_{\p},q)}{q-1},d\right\}.
\end{equation}

Determining when we have equality in \eqref{similarity dimension inequality}, \eqref{similarity dimension for measures}, and \eqref{Lq upper bound} is an active and important area of research (see \cite{Hochman2, Hochman,Hochman3,Rap,Shm,Shm2,Varju} and the references therein). When $\Phi$ is an IFS acting on $\mathbb{R}$ the only known mechanism preventing equality in \eqref{similarity dimension inequality}, \eqref{similarity dimension for measures}, or \eqref{Lq upper bound} is the 
existence of distinct $\a,\b\in \cup_{n=1}^{\infty}\I^n$ such that $\phi_{\a}=\phi_{\b}.$ When such an $\a$ and $\b$ exists we say that $\Phi$ contains an exact overlap. Here and throughout we adopt the notational convention that $\phi_{\a}:=\phi_{a_1}\circ\cdots \circ \phi_{a_n}$ for $\a=(a_1,\ldots,a_n)$. In higher dimensions there are other mechanisms preventing equality. In particular an IFS could force the self-similar set into a lower dimensional affine subspace of $\R^d$ without it containing an exact overlap. For more on the mechanisms preventing equality in higher dimensions see \cite{Hochman3}. For iterated function systems acting on $\R,$ it is conjectured that the only mechanism preventing equality in \eqref{similarity dimension inequality}, \eqref{similarity dimension for measures}, and \eqref{Lq upper bound} is the presence of an exact overlap. This conjecture is commonly referred to as the exact overlaps conjecture. Significant progress on this conjecture has been made in recent years. Rapaport in \cite{Rap} proved that if $\Phi$ is an IFS acting on $\R$ with algebraic similarity ratios, then either $\Phi$ contains an exact overlap or we have equality in \eqref{similarity dimension inequality} and \eqref{similarity dimension for measures} for all self-similar measures. Building upon the work of Hochman in \cite{Hochman}, Varju recently proved that for the unbiased Bernoulli convolution we have equality in \eqref{similarity dimension for measures} when the underlying IFS does not contain an exact overlap \cite{Varju}. The motivation behind this work comes from two important results on this topic due to Hochman \cite{Hochman} and Shmerkin \cite{Shm2}. Combining their results it follows that if $\Phi$ is an IFS acting on $\R$ and we have strict inequality in one of \eqref{similarity dimension inequality}, \eqref{similarity dimension for measures}, or \eqref{Lq upper bound}, then $\Phi$ comes extremely close to containing an exact overlap. This notion of closeness is provided by the following distance function. Given two contracting similarities $\phi(x)=rx +t$ and $\phi'(x)=r'x+t',$ we let $$d(\phi,\phi'):=\begin{cases}
\infty,\, &r\neq r'\\
|t-t'|,\, &r=r'.
\end{cases}$$ Importantly $d(\phi,\phi')=0$ if and only if $\phi=\phi'.$ Given an IFS $\Phi$, for any $n\in \N$ we let $$\Delta_n(\Phi):=\min\{d(\phi_{\a},\phi_{\b}):\, \a,\b\in \I^n,\, \a\neq \b\}.$$ The results of Hochman and Shmerkin can now be more accurately summarised as follows. 
\begin{thm}
	\label{HocShm}Let $\Phi$ be an IFS acting on $\R$.  
	\begin{itemize}
		\item \cite[Theorem 1.1]{Hochman} If $\limsup_{n\to\infty}\frac{\log \Delta_n}{n}>-\infty$ then we have equality in \eqref{similarity dimension inequality} and \eqref{similarity dimension for measures} for all self-similar measures.
		\item \cite[Theorem 6.6]{Shm2} If $\limsup_{n\to\infty}\frac{\log \Delta_n}{n}>-\infty$ then we have equality in \eqref{Lq upper bound} for all self-similar measures and $q>1$.
	\end{itemize}
\end{thm}

With Theorem \ref{HocShm} and the exact overlaps conjecture in mind, it is natural to ask whether there exists an IFS $\Phi,$ such that $\Phi$ does not contain an exact overlap yet $\Delta_n(\Phi)\to 0$ super-exponentially fast. This question was posed by Hochman in \cite{Hochman2}. If no such $\Phi$ exists then Theorem \ref{HocShm} would imply the exact overlaps conjecture. Using ideas from \cite{Bak} it was shown in \cite{Bak2} that such $\Phi$ do exist. Interestingly it was shown that $(\Delta_n(\Phi))_{n=1}^{\infty}$ can converge to zero arbitrarily fast without $\Phi$ containing an exact overlap. At the same time and using a different method B\'{a}r\'{a}ny and K\"{a}enm\"{a}ki obtained the same result \cite{BarKae}. Recently Chen \cite{Chen} altered the construction given in \cite{Bak2} to allow for algebraic contraction ratios. 

In this paper we prove a general result for the existence of an IFS $\Phi$ within a parameterised family of IFSs such that $(\Delta_{n}(\Phi))_{n=1}^{\infty}$ converges to zero arbitrarily fast and $\Phi$ does not contain an exact overlap. This serves several purposes. First of all, the argument given is more general and intuitive then the one presented in \cite{Bak2}. The argument shows that if a parameterised family of IFSs contains two independent subfamilies, then it is reasonable to expect that the original family will contain a $\Phi$ satisfying our desired properties. Secondly, our more general result provides new examples of $\Phi$ satisfying these properties. In the final section of this paper we include several examples of families to which our result can be applied. 



\section{Preliminaries and our main result}

In the statement of our main result we will be working in $\R^d$ and so require a higher dimensional analogue of $d(\cdot,\cdot)$ and $\Delta_n(\cdot)$. It is a well known fact that any contracting similarity $\phi$ acting on $\R^d$ can be uniquely written as $\phi(\x)=r\cdot O\x+t$ for some $r\in(0,1)$, $O\in \O(d),$ and $\t\in\R^d$. Here $\O(d)$ denotes the group of $d\times d$ orthogonal matrices. As such, given two contracting similarities $\phi(\x)=r\cdot O\x+t$ and $\phi'(\x)=r'O'\x+t'$ the following quantity is well defined $$d(\phi,\phi')=|t-t'|+|\log r - \log r'|+\|O-O'\|.$$ Here $\|\cdot\|$ denotes the operator norm. This is the distance function defined by Hochman in \cite{Hochman3}. Importantly it has the property that $d(\phi,\phi')=0$ if and only if $\phi=\phi'$. Given an IFS $\Phi$ and $n\in \N$ we define $$\Delta_{n}(\Phi):=\min\{d(\phi_{\a},\phi_{\b}):\a,\b\in \I^n, \a\neq \b\}.$$
Note that the function $d(\cdot,\cdot)$ behaves differently to the distance function defined previously for similarities acting on $\R.$ With this new distance function it is possible for $d(\phi,\phi')$ to take small values for $\phi$ and $\phi'$ with different similarity ratios. If we equip $\O(d)$ with the topology obtained under the usual identification with a subset of $\R^{d\times d},$ and identify the space of contracting similarities with $(0,1)\times \O(d)\times \R^d$, then $d(\cdot,\cdot)$ is a continuous function from the space of pairs of similarities into $[0,\infty)$. This fact will be important in our proofs.

Our main result will be phrased in terms of parameterised families of IFSs. The general framework we use for such a parameterisation is as follows. Let $U\subset \R^{k_1}$ and $V\subset \R^{k_2}$ be open subsets of their respective Euclidean spaces. Let $\I_1,$ $\I_2$ and $\I_3$ be finite sets. We assume that for each $i\in \I_1$  there exists continuous functions $O_{i,1}:U\to \O(d)$, $r_{i,1}:U\to (0,1),$ and $t_{i,1}:U\to \R^d$. Similarly, for  each $i\in \I_2$ we assume that there exists continuous functions $O_{i,2}:V\to \O(d)$, $r_{i,2}:V\to (0,1),$ and $t_{i,2}:V\to \R^d$. Also, for each $i\in I_3$ we assume that there exists continuous functions $O_{i,3}:U\times V\to \O(d)$, $r_{i,3}:U\times V\to (0,1),$ and $t_{i,3}:U\times V\to \R^d.$ Equipped with these functions we can define three parameterised families of iterated function systems. Given $u\in U$ we let $$\Phi_{u}:=\{\phi_{i,u}(x)=r_{i,1}(u)\cdot O_{i,1}(u) x +t_{i,1}(u)\}_{i\in \I_1},$$ and similarly given $v\in V$ we let $$\Phi_{v}:=\{\phi_{i,v}(x)=r_{i,2}(v)\cdot O_{i,2}(v) x +t_{i,2}(v)\}_{i\in \I_2}.$$ Given $(u,v)\in U\times V$ we define $$\Phi_{u,v}:=\Phi_{u}\cup \Phi_{v}\cup \{\phi_{i,u,v}(x)=r_{i,3}(u,v)\cdot O_{i,3}(u,v)x+t_{i,3}(u,v)\}_{i\in \I_3}.$$ Note that all three of these iterated function systems are acting on $\R^d$. In our applications we may simply take $\I_3=\emptyset$ and $\Phi_{u,v}=\Phi_{u}\cup \Phi_{v}.$

We define $$H_{1}(n):=\{u\in U:\phi_{\a,u}=\phi_{\b,u} \textrm{ for some }\a,\b\in \I_1^{n}, \a\neq \b\}$$ and $$H_{1}^{*}(n):=H_1(n)\setminus \cup_{j=1}^{n-1}H_{1}(j).$$ We use $H_{2}(n)$ and $H_{2}^{*}(n)$ to denote the corresponding sets for the family $\{\Phi_v\}_{v\in V},$ and $H_3(n)$ and $H_3^{*}(n)$ to denote the corresponding sets for the family $\{\Phi_{u,v}\}_{(u,v)\in U\times V}.$  

The following statement is the main result of this paper.
\begin{thm}
	\label{Main thm}
	Let $\{\Phi_u\}_{u\in U},\{\Phi_v\}_{v\in V},$ and $\{\Phi_{u,v}\}_{(u,v)\in U\times V}$ be parameterised families of iterated function systems as defined above. Suppose that the following properties are satisfied:
	\begin{enumerate}
			\item $\cup_{n=1}^{\infty}H_{1}(n)\neq \emptyset$ and $\cup_{n=1}^{\infty}H_{2}(n)\neq \emptyset.$
			\item For any $u\in H_{1}(n_0),$ $v\in H_{2}(m_0),$ and $\epsilon>0$, there exists $u_1\in \cup_{n=\max\{n_0,m_0\}+1}^{\infty}H_{1}^*(n)$ such that:
			\begin{itemize}
				\item[(a)] $\|u-u_1\|_{\infty}<\epsilon.$
				\item[(b)] For any $\epsilon'>0$, there exists $v_1\in \cup_{n=\max\{n_0,m_0\}+1}^{\infty}H_{2}^*(n)$ such that $\|v-v_1\|_{\infty}<\epsilon'$ and $(u_1,v_1)\notin H_{3}(\max\{n_0,m_0\}).$
			\end{itemize} 
		\end{enumerate} Then for any sequence $(\omega_n)_{n=1}^{\infty}$ of strictly positive real numbers, there exists $(u^*,v^*)\in U\times V$ such that $\Delta_n(\Phi_{u^*,v^*})\leq \omega_n$ for all $n$ sufficiently large and $\Phi_{u^*,v^*}$ contains no exact overlaps.
\end{thm}
We emphasise that the parameter $\epsilon'$ appearing in $2b.$ can be chosen to depend upon $u_1$. We will use this fact in our proof. 

Theorem \ref{Main thm} can be used to recover and strengthen the results of \cite{Bak2} and \cite{Chen}, see Example \ref{Example 1}. Importantly the argument given bypasses the need to rely on properties of continued fractions. The following corollary is also implied by this theorem.


\begin{cor}
\label{corollary}
Let $\{\Phi_u\}_{u\in U},\{\Phi_v\}_{v\in V}$ and $\{\Phi_{u,v}\}_{(u,v)\in U\times V}$ be parameterised families of iterated function systems as defined above. Suppose that the following properties are satisfied:
\begin{enumerate}
\item $\cup_{n=1}^{\infty}H_{1}(n)$ and $\cup_{n=1}^{\infty}H_{2}(n)$ are dense in $U$ and $V$ respectively. 
\item For any $n\in \N$ the sets $H_{1}(n)$ and $H_{2}(n)$ are both nowhere dense.
\item Let $n_0<n_1.$ For any $u\in H_{1}^{*}(n_1)$ the set $$\{v\in V:\phi_{\a,u,v}=\phi_{\b,u,v}\textrm{ for some }\a,\b\in \I_{3}^{n_0}, \a\neq \b\}$$ is nowhere dense.
\end{enumerate}
 Then for any sequence $(\omega_n)_{n=1}^{\infty}$ of strictly positive real numbers, there exists $(u^*,v^*)\in U\times V$ such that $\Delta_n(\Phi_{u^*,v^*})\leq \omega_n$ for all $n$ sufficiently large and $\Phi_{u^*,v^*}$ contains no exact overlaps.
\end{cor}
The hypotheses appearing in Corollary \ref{corollary} are natural. For many parameterised families of overlapping IFSs the set of parameters causing exact overlaps are dense. For any $n\in \N$ the set $H_1(n)$ is closed, and so if it failed to be nowhere dense then there would be an non-empty open subset contained in $H_1(n)$. Such a set would mean that there exists a sizeable part of $U$ for which the corresponding IFSs have overlaps of length $n$ effectively built in. The same is true for $H_2(n)$. Similarly, if the third condition was not satisfied it would mean that there exists $u'\in H_{1}^{*}(n_1)$ for which the parameterised family of IFSs given by $\{\Phi_{u',v}\}_{v\in V}$ would effectively have exact overlaps of length $n_0$ built in within some non-empty open subset of $V$.

\subsection{Proof of Theorem \ref{Main thm}}
The following lemma records several elementary facts that we will need in our proof. These facts follow immediately from the definitions and their proofs are omitted.
\begin{lemma}
	\label{useful lemma}
	\begin{enumerate}
		\item For any IFS $\Phi$ the sequence $(\Delta_n(\Phi))_{n=1}^{\infty}$ is decreasing.
		\item Let $\Phi_1,\Phi_2,$ and $\Phi_3$ be three IFSs satisfying $\Phi_1\subseteq \Phi_3$ and $\Phi_2\subseteq \Phi_3$. Then for any $n\in \N$ we have $$\Delta_n(\Phi_{3})\leq \min \{\Delta_n(\Phi_1),\Delta_n(\Phi_2)\}.$$
		\item For any $n\in \N$ the sets $H_{1}(n)$, $H_{2}(n),$ and $H_3(n)$ are closed.
	\end{enumerate}
\end{lemma}

The proof of Theorem \ref{Main thm} relies on a simple strategy that we outline here. Given a sequence of positive real numbers $(\omega_n)$, it is a relatively easy task to construct an element $u^{*}\in U$ such that $\Delta_{n}(\Phi_{u^*})\leq \omega_n$ for all $n$ belonging to infinitely many long stretches of the natural numbers. This can be achieved by requiring $u^*$ be extremely well approximated at infinitely many scales by parameters that cause exact overlaps. The issue here is that there might exist gaps between those stretches of natural numbers for which $\Delta_{n}(\Phi_{u^*})\leq \omega_n$ holds. The trick is to use the family $\{\Phi_{v}\}_{v\in V}$ to find an element $v^*\in V$ for which the set of $n$ for which $\Delta_{n}(\Phi_{v^*})\leq \omega_n$ fills in these gaps. Ensuring $\Phi_{u^{*},v^*}$ contains no exact overlaps can be achieved by a topological argument.

In our proof we let $B(x,r)$ denote the open ball of radius $r$ centred at $x$ with respect to the infinity norm.

\begin{proof}[Proof of Theorem \ref{Main thm}]
Let us start by fixing $(\omega_n)_{n=1}^{\infty}$ a sequence of strictly positive real numbers. The element $(u^*,v^*)$ we construct will belong to the countable intersection of a nested collection of closed balls. We define these balls below via an inductive argument.\\

\noindent \textbf{Base case.} Pick $u_0\in H_{1}(n_0)$ and $v_0\in H_{2}(m_0)$ for some $n_0, m_0\in \N$. Both $u_0$ and $v_0$ exist by property $1.$ By definition there exists distinct $\a_0,\b_0\in \I_1^{n_0}$ such that $\phi_{\a_0,u_0}=\phi_{\b_0,u_0}$. Since the maps $U_{i,1}$, $r_{i,1}$, and $t_{i,1}$ are continuous, and the distance function $d$ is continuous, there exists $\epsilon_0>0$ such that \begin{equation*}
B(u_0,\epsilon_0)\subseteq \left\{u:d(\phi_{\a_0,u},\phi_{\b_0,u})\leq \min_{n_0\leq n\leq \max\{n_0,m_0\}}\omega_n\right\}.
\end{equation*} It follows from the definition of $\Delta_n(\cdot)$ and Lemma \ref{useful lemma}.1 that \begin{equation}
\label{inclusion1}
B(u_0,\epsilon_0)\subseteq\left\{u:\Delta_n(\Phi_{u})\leq \omega_n \textrm{ for all }n_0\leq n\leq \max\{n_0,m_0\}\right\}.
\end{equation} Let $u_1\in U$ be the element whose existence is guaranteed by property $2.$ for the choice of parameters $u_0,v_0$ and $\epsilon_0.$ Therefore $u_1\in H_{1}^*(n_1)$ for some $n_1>\max\{n_0,m_0\}$ and $u_1\in B(u_0,\epsilon_0)$.


Since $v_0\in H_{2}(m_0)$ there exists distinct $\a_0',\b_0'\in \I_{2}^{m_0}$ such that $\phi_{\a_0',v_0}=\phi_{\b_0',v_0}.$ By the same reasoning as above, we may pick $\epsilon_0'>0$ such that \begin{equation}
\label{inclusion2}
B(v_0,\epsilon_0')\subseteq \left\{v:\Delta_n(\Phi_{v})\leq \omega_n \textrm{ for all }m_0\leq n\leq n_1\right\}.
\end{equation} Since $u_1$ is the point whose existence is guaranteed by property $2$, we know that there exists $v_1\in H_{2}^*(m_{1})$ for some $m_1>\max\{n_0,m_0\}$ such that $v_1\in B(v_0,\epsilon_0')$ and $(u_1,v_1)\notin H_{3}(\max\{n_0,m_0\})$. By Lemma \ref{useful lemma}.3 we know that $H_{3}(\max\{n_0,m_0\})$ is closed. Therefore there exists sufficiently small $\delta_1>0$ such that 
\begin{equation}
\label{empty1}
\overline{B((u_1,v_1),\delta_1)}\cap H_{3}(\max\{n_0,m_0\})=\emptyset.
\end{equation}Moreover, since $u_1\in B(u_0,\epsilon_0)$ and $v_1\in B(v_0,\epsilon_0),$ we may also assume that $\delta_1$ is sufficiently small that
\begin{equation*}
\label{Super close level 1}
\overline{B((u_1,v_1),\delta_1)}\subseteq B(u_0,\epsilon_0)\times B(v_0,\epsilon_0').
\end{equation*} 
Therefore by \eqref{inclusion1}, \eqref{inclusion2}, and Lemma \ref{useful lemma}.2, we may deduce that $$\Delta_n(\Phi_{u,v})\leq \omega_n \textrm{ for all }n_0\leq n\leq n_1 $$ for all $(u,v)\in \overline{B((u_1,v_1),\delta_1)}.$\\


\noindent \textbf{Inductive step.} Suppose that we have constructed sequences $(u_k)_{k=0}^{K},$ $(v_k)_{k=0}^K$, $(n_k)_{k=0}^K$, $(m_k)_{k=0}^{\infty}$, and $(\delta_k)_{k=1}^K$ such that the following properties are satisfied:
\begin{enumerate}
	\item[(a)] $u_k\in H_{1}^*(n_k)$ and $v_k\in H_{2}^{*}(m_k)$ for all $1\leq k\leq K$.
	\item[(b)] Both $(n_k)_{k=0}^K$ and $(m_k)_{k=0}^{K}$ are strictly increasing sequences.
	\item[(c)] $(\overline{B((u_k,v_k),\delta_k)})_{k=1}^{K}$ is a nested sequence.
	\item[(d)]$ \overline{B((u_k,v_k),\delta_k)}\cap H_{3}(\max\{n_{k-1},m_{k-1}\})=\emptyset$ for all $1\leq k\leq K$.
	\item[(e)] For all $(u,v)\in \overline{B((u_K,v_K),\delta_K)}$ we have
	$$\Delta_{n}(\Phi_{u,v})\leq \omega_n \textrm{ for all }n_0\leq n\leq n_{K}.$$
\end{enumerate}In the base case we constructed the relevant sequences for $K=1$. We now show how these sequences are defined for $K+1$.\\

By definition there exists distinct $\a_K,\b_K\in \I_{1}^{n_K}$ such that $\phi_{\a_K,u_K}=\phi_{\b_K,u_K}.$ Therefore by analogous reasoning to that given in the base case, there exists $\epsilon_{K}>0$ such that
\begin{equation}
\label{inclusionK}
B(u_K,\epsilon_{K})\subseteq\left\{u:\Delta_n(\Phi_{u})\leq \omega_n \textrm{ for all }n_k\leq n\leq \max\{n_K,m_K\}\right\}\cap B(u_K,\delta_K).
\end{equation} 
Let $u_{K+1}\in U$ be the element whose existence is guaranteed by property $2.$ for the choice of parameters $u_{K},v_{K}$, and $\epsilon_{K}.$ Therefore $u_{K+1}\in H_{1}^{*}(n_{K+1})$ for some $n_{K+1}>\max\{n_{K},m_{K}\}$ and $u_{K+1}\in B(u_{K},\epsilon_{K}).$

Since $v_K\in H_{2}^*(m_K)$ there exists distinct $\a_{K}',\b_{K}'\in \I_{2}^{m_k}$ such that $\phi_{\a_{K}',v_{K}}=\phi_{\b_{K}',v_{K}}.$ By analogous reasoning to that given in the base case, it follows that there exists $\epsilon_{K}'>0$ such that 
\begin{equation}
\label{inclusionk2}
B(v_K,\epsilon_{K}')\subseteq\left\{v:\Delta_n(\Phi_{v})\leq \epsilon_n \textrm{ for all }m_k\leq n\leq n_{k+1}\right\}\cap B(v_K,\delta_K).
\end{equation}By the definition of $u_{K+1},$ we know that there exists $v_{K+1}\in H_{2}^*(m_{K+1})$ for some $m_{K+1}>\max\{n_{K},m_{K}\}$ such that $v_{K+1}\in B(v_K,\epsilon_K')$ and $(u_{K+1},v_{K+1})\notin H_{3}(\max\{n_K,m_K\}).$ By Lemma \ref{useful lemma}.3 we know that the set $H_{3}(\max\{n_K,m_K\})$ is closed. Therefore there exists $\delta_{K+1}>0$ such that $$\overline{B((u_{K+1},v_{K+1}),\delta_{K+1})}\cap H_{3}(\max\{n_K,m_K\})=\emptyset.$$ Since $u_{K+1}\in B(u_K,\epsilon_K)$ and $v_{K+1}\in B(v_K,\epsilon_K')$ we may also assume that $\delta_{K+1}$ is sufficiently small so that 
\begin{equation}
\label{extra inclusion}
\overline{B((u_{K+1},v_{K+1}),\delta_{K+1})}\subseteq B(u_K,\epsilon_K)\times B(v_K,\epsilon_K').
\end{equation} 
It follows from Lemma \ref{useful lemma}.2, \eqref{inclusionK}, \eqref{inclusionk2}, and \eqref{extra inclusion} that if $(u,v)\in \overline{B((u_{K+1},v_{K+1}),\delta_{K+1})}$ then $$\Delta_{n}(\Phi_{u,v})\leq \epsilon_n \textrm{ for }n_K\leq n\leq n_{K+1}.$$ Moreover, it follows from \eqref{inclusionK}, \eqref{inclusionk2}, and \eqref{extra inclusion} that $\overline{B((u_{K+1},v_{K+1}),\delta_{K+1})}\subseteq \overline{B((u_{K},v_{K}),\delta_{K})}$. Therefore by property $(e)$ we may conclude that $$\Delta_n(\Phi_{u,v})\leq \epsilon_n \textrm{ for }n_0\leq n\leq n_{K+1}$$ for all $(u,v)\in \overline{B((u_{K+1},v_{K+1}),\delta_{K+1})}$. For these choices of $u_{K+1},v_{K+1},n_{K+1},m_{K+1}$ and $\delta_{K+1}$ we see that properties $(a)-(e)$ are satisfied for level $K+1$. This completes our inductive step.\\ 

\noindent \textbf{Conclusion of proof.} Repeating the inductive step indefinitely we can define sequences $(u_k)_{k=0}^{\infty},$ $(v_k)_{k=0}^{\infty}$, $(n_k)_{k=0}^{\infty}$, $(m_k)_{k=0}^{\infty}$, and $(\delta_k)_{k=1}^{\infty}$ such that properties $(a)-(e)$ hold for any choice of $K$.

Since  $(\overline{B((u_k,v_k),\delta_k)})_{k=1}^{\infty}$ is a nested sequence of closed balls we have $$\bigcap_{k=1}^{\infty}\overline{B((u_k,v_k),\delta_k)}\neq \emptyset.$$ Taking $(u^*,v^*)$ in this intersection it follows from property $(e)$ that $\Delta_{n}(\Phi_{u^*,v^*})\leq \omega_n$ for all $n\geq n_0$. Moreover by property $(d)$ we may also conclude that $(u^*,v^*)\notin H_{3}(\max\{n_k,m_k\})$ for any $k\in \mathbb{N}$. Since $(n_k)_{k=1}^{\infty}$ and $(m_k)_{k=1}^{\infty}$ are both strictly increasing sequences and $H_3(n)\subseteq H_{3}(n')$ for $n\leq n',$ it follows that $(u^*,v^*)\notin H_{3}(n)$ for any $n$. Therefore $\Phi_{u^*,v^*}$ also contains no exact overlaps as required.


\end{proof}
\begin{remark}
	Adapting the proof of Theorem \ref{Main thm} it can be shown that the set $$\left\{(u,v):\Phi_{u,v}\textrm{ contains no exact overlaps and }\Delta_n(\Phi_{u,v})\leq \omega_n \textrm{ for all }n \textrm{ sufficiently large}\right\}$$ is dense in $U\times V$ under the additional assumption that both $\cup_{n=1}^{\infty}H_{1}(n)$ and $\cup_{n=1}^{\infty}H_{2}(n)$ are dense in $U$ and $V$ respectively.
\end{remark}
\begin{remark}
In the statement of Theorem \ref{Main thm} it is possible to strengthen the assertion $\Delta_n(\Phi_{u^*,v^*})\leq \omega_n$ for all $n$ sufficiently large. A careful inspection of the proof shows that we can assert that there exists $(u^*,v^*)$ such that $\Delta_n(\Phi_{u^*,v^*})\leq \omega_n$ for all $n\geq \min\{n:H_{1}(n)\neq \emptyset\}.$
\end{remark}

\section{Examples}
In this section we detail two examples of parameterised families of IFSs to which we can apply Theorem \ref{Main thm}. Our first family does not impose any algebraic conditions on the similarity ratio and this quantity can be taken to be transcendental. This is particularly relevant given the aforementioned result of Rapaport \cite{Rap} which establishes the exact overlaps conjecture when the underlying contraction ratios are all algebraic. The second family is inspired by the Bernoulli convolution and establishes the existence of a non-equicontractive IFS with super-exponentially close cylinders and no exact overlaps.
 
Note that the function $\Delta_n(\cdot)$ appearing in Theorem \ref{Main thm} is defined using a distance function for which similarities with different contraction ratios can be close. For the examples given below we may insist $\Delta_n(\cdot)$ be defined using the following stricter distance function and still guarantee the existence of $\Phi_{\u^*,\v^*}$ without exact overlaps such that $(\Delta_n(\Phi_{\u^*,\v^*}))_{n=1}^{\infty}$ converges to zero at any desired speed. Given two contracting similarities $\phi(\x)=r\cdot O\x +t$ and $\phi'(\x)=r'\cdot O'\x+t'$ acting on $\mathbb{R}^d$ we define our stricter distance function via the equation $$d(\phi,\phi'):=\begin{cases}
\infty,\, &r\neq r' \textrm{ or } O\neq O'\\
|t-t'|,\, &r=r'\textrm{ and }O=O'.
\end{cases}$$

\begin{example}
	\label{Example 1}
In this example $\lambda\in(0,1/2]$ is some fixed parameter. We could take $U=V=\mathbb{R}^d$, however for ease of exposition it is more convenient to set $$U=V:=\{(x_1,\ldots,x_d)\in \R^d: x_j\notin\{0,1\}\,\textrm{ for all } 1\leq j\leq d\}.$$ Given a vector $\u=(u_1,\ldots, u_d)\in U$ we associate the IFS $$\Phi_{\u}:=\left\{\lambda(\x +\a):\a\in \prod_{j=1}^d\{0,1,u_j,1+u_j\}\right\}.$$ Similarly, given $\v=(v_1,\ldots,v_d)\in V$ we let $$\Phi_{\v}:=\left\{\lambda( \x +\a):\a\in \prod_{j=1}^d\{0,1,v_j,1+v_j\}\right\}.$$ Moreover, given $\u$ and $\v$ let $\mathbf{T}_{\u,\v}\subset \R^d$ be any finite set of vectors satisfying $$\prod_{j=1}^d\{0,1,u_j,1+u_j\}\cup \prod_{j=1}^d\{0,1,v_j,1+v_j\}\subseteq \mathbf{T}_{u,v}\subseteq \prod_{j=1}^d\{0,1,u_j,v_j,1+u_j,1+v_j,u_j+v_j, 1+u_j+v_j\}.$$ Let $$\Phi_{\u,\v}=\left\{\lambda(\x+\a):\a\in \mathbf{T}_{\u,\v}\right\}.$$ Then $\{\Phi_{\u}\}, \{\Phi_{\v}\}$ and $\{\Phi_{\u,\v}\}$ are three parameterised families of IFSs as considered in Theorem \ref{Main thm}.

After a little calculation it can be shown that $\u\in H_{1}(n)$ if and only if\footnote{Here we use our assumption $\lambda\in (0,1/2]$.} $$\u\in \bigcup_{j=1}^d\left\{\u\in U: u_j=\frac{\sum_{i=1}^n\kappa_i\lambda^i}{\sum_{i=1}^n \delta_i\lambda^i},\, (\kappa_i)_{i=1}^n, (\delta_i)_{i=1}^n\in\{-1,0,1\}^n\setminus\{(0)^n\}\right\}.$$ Obviously the same holds for $H_{2}(n)$ with the parameter $\v$ replacing $\u.$ This shows that $H_{1}(n)$ and $H_{2}(n)$ are both non-empty for any $n\in \N$. Therefore property $1.$ from Theorem \ref{Main thm} holds for these families. 

To see that property $2$ from Theorem \ref{Main thm} holds for this family we fix $\u\in H_{1}(n_0),$ $\v\in H_{2}(m_0)$ and $\epsilon>0$. Let $L_0:=\max\{n_0,m_0\}$ and $$\mathcal{D}(\u)=\left\{j:u_j=\frac{\sum_{i=1}^{L_0}\kappa_i\lambda^i}{\sum_{i=1}^{L_0} \delta_i\lambda^i},\, \textrm{ for some } (\kappa_i)_{i=1}^{L_0}, (\delta_i)_{i=1}^{L_0}\in\{-1,0,1\}^{L_0}\setminus\{(0)^{L_0}\}\right\}.$$ $D(\u)$ is the set of coordinate positions of $\u$ that give rise to an exact overlap caused by two distinct words of length at most $L_0$. Since $\u\in H_{1}(n_0)$ it follows from the characterisation of $H_{1}(n_0)$ above that the set $D(\u)$ is non-empty. For each $N\in\N$, if $j\in \mathcal{D}(\u)$ and $$u_j=\frac{\sum_{i=1}^{L_0}\kappa_i\lambda^i}{\sum_{i=1}^{L_0} \delta_i\lambda^i}$$ then we define $$u_{j,N}:=\frac{\sum_{i=1}^{L_0}\kappa_i\lambda^i+\lambda^{N}}{\sum_{i=1}^{L_0} \delta_i\lambda^i},$$ and for all $j\notin \mathcal{D}(\u)$ we set $u_{j,N}:=u_j$. For each $N\in\N$ we then let $\u_{N}:=(u_{1,N},\ldots,u_{d,N})$. The set $$\left\{\frac{\sum_{i=1}^{L_0}\kappa_i\lambda^i}{\sum_{i=1}^{L_0} \delta_i\lambda^i},\, \textrm{ for some } (\kappa_i)_{i=1}^{L_0}, (\delta_i)_{i=1}^{L_0}\{-1,0,1\}^{L_0}\setminus\{(0)^{L_0}\}\right\}$$ is finite. Therefore for any $N$ sufficiently large 
the parameter $u_{j,N}$ does not belong to this set for any $j$. It follows from the characterisation of $H_{1}(n)$ stated above that $\u_{N}\notin H_{1}(L_0)$ for $N$ sufficiently large. Using the characterisation of $H_{1}(n)$ stated above again, and the fact $H_{1}(L_0)=\cup_{n=1}^{L_0}H_{1}(n),$ it follows that $\u_N\in \cup_{n=L_0+1}^{\infty}H_{1}^*(n)$ for $N$ sufficiently large. Moreover, we may also assume that $N$ is sufficiently large that $\u_{N}$ satisfies $\|\u-\u_{N}\|_{\infty}<\epsilon.$  We define $\u_1=\u_{N}$ for any $N$ sufficiently large such that both of these properties hold.

Let $\u_1=(u_{1,1},\ldots, u_{d,1})$. We've shown that $\u_1$ satisfies $2a$ from Theorem \ref{Main thm}. To show that property $2b$ holds for this choice of $\u_1$ let us fix $\epsilon'>0$. By a simple calculation, it can be shown that for $\v^*\in V$, if $(\u_1,\v^*)\in H_{3}(L_0)$ then 
\begin{align*}\v^*\in \bigcup_{j=1}^{d}\Bigg\lbrace\v\in V: v_j=\frac{\sum_{i=1}^{L_0}\kappa_i\lambda^i+u_{j,1}\sum_{i=1}^{L_0}\delta_i\lambda^i}{\sum_{i=1}^{L_0}\gamma_i\lambda^i},&\text{ for some } (\gamma_i)\in\{-1,0.1\}^{L_0}\setminus\{(0)^{L_0}\}\\
&\textrm{ and }(\kappa_i),(\delta_i)\in\{-1,0,1\}^{L_0} \Bigg\rbrace.
\end{align*} For our previously fixed value of  $\v$ let
$$\mathcal{D}'(\v)=\left\{j:v_j=\frac{\sum_{i=1}^{L_0}\kappa_i\lambda^i}{\sum_{i=1}^{L_0} \delta_i\lambda^i} \textrm{ for some }(\kappa_i),(\delta_i)\in\{-1,0,1\}^{L_0}\setminus\{(0)^{L_0}\}\right\}$$ 
and
$$\mathcal{D}''(\v)=\left\{j:v_j=\frac{\sum_{i=1}^{L_0}\kappa_i\lambda^i+u_{j,1}\sum_{i=1}^{L_0}\delta_i\lambda^i}{\sum_{i=1}^{L_0}\gamma_i\lambda^i}, (\gamma_i)\in\{-1,0,1\}^{L_0}\setminus\{(0)^{L_0}\}, (\kappa_i),(\delta_i)\in\{-1,0,1\}^{L_0}\right\}.$$ Since $\v\in H_{2}(m_0)$ it follows that $\v\in H_{2}(L_0)$ and therefore $\mathcal{D}'(\v)$ is non-empty. 

For each $N\in \N,$ if $j\in \mathcal{D'}(\v)$ and $$v_{j}=\frac{\sum_{i=1}^{L_0}\kappa_i\lambda^i}{\sum_{i=1}^{L_0} \delta_i\lambda^i},$$ then we define 
$$v_{j,N}:=\frac{\sum_{i=1}^{L_0}\kappa_i\lambda^i+\lambda^{N}}{\sum_{i=1}^{L_0} \delta_i\lambda^i},$$
if $j\in \D''(\v)\setminus \D'(\v)$ and $$v_j=\frac{\sum_{i=1}^{L_0}\kappa_i\lambda^i+u_{j,1}\sum_{i=1}^{L_0}\delta_i\lambda^i}{\sum_{i=1}^{L_0}\gamma_i\lambda^i},$$then define $$v_{j,N}:=\frac{\sum_{i=1}^{L_0}\kappa_i\lambda^i+\lambda^{N}+u_{j,1}\sum_{i=1}^{L_0}\delta_i\lambda^i}{\sum_{i=1}^{L_0}\gamma_i\lambda^i},$$ and if $j\notin \mathcal{D}'(\v)\cup\mathcal{D}''(\v)$ then $$v_{j,N}:=v_j.$$ For each $N\in \N$ set $\v_{N}:=(v_{1,N},\ldots,v_{j,N})$.

The sets $$\left\{\frac{\sum_{i=1}^{L_0}\kappa_i\lambda^i}{\sum_{i=1}^{L_0} \delta_i\lambda^i},(\kappa_i), (\delta_i)\in\{-1,0,1\}^{L_0}\setminus\{(0)^{L_0}\}\right\}$$ and $$\left\{\frac{\sum_{i=1}^{L_0}\kappa_i\lambda^i+u_{j,1}\sum_{i=1}^{L_0}\delta_i\lambda^i}{\sum_{i=1}^{L_0}\gamma_i\lambda^i}, (\gamma_i)\in\{-1,0.1\}^{L_0}\setminus\{(0)^{L_0}\}, (\kappa_i), (\delta_i)\in\{-1,0,1\}^{L_0}\right\}$$ are both finite. Therefore for any $N$ sufficiently large, $v_{j,N}$ does not belong to either of these sets for any $j$. Therefore $\v_{N}\in \cup_{n=L_0+1}^{\infty}H_{2}^*(n)$ and $(\u_1,\v_{N})\notin H_3(L_0)$ for $N$ sufficiently large. Moreover, for $N$ sufficiently large we also have $\|\v-\v_{N}\|_{\infty}<\epsilon'$. We take $\v_{1}=\v_{N}$ for any $N$ sufficiently large such that each of these properties hold. For this choice of $\v_1$ we see that property $2b$ from Theorem \ref{Main thm} is satisfied. As such this theorem can be applied to conclude the existence of an IFS $\Phi_{\u^*,\v^*}$ within this family without exact overlaps such that $(\Delta_n(\Phi_{\u^*,\v^*}))$ converges to zero arbitrarily fast.

\end{example}

\begin{example}
In this example $U=V=(1/2,1).$ For $u\in (1/2,1)$ we let $$\Phi_{u}:=\{u(x+1),u(x+2)\},$$ and for $v\in (1/2,1)$ we let $$\Phi_{v}:=\{v(x+1),v(x+2)\}.$$ Moreover, for $(u,v)\in (1/2,1)\times (1/2,1)$ let $$\Phi_{u,v}:=\{u(x+1),u(x+2),v(x+1),v(x+2)\}.$$ It is straightforward to show that $u\in H_{1}(n)$ if and only if $$u\in \left\{u:\sum_{i=1}^n\kappa_iu^i=0 \textrm{ for some }(\kappa_i)\in\{-1,0,1\}^n\setminus \{(0)^n\}\right\}.$$ An equivalent characterisation exists for $H_{2}(n)$ with $u$ replaced by $v$. The set of zeros of non-zero polynomials with coefficients belonging to the set $\{-1,0,1\}$ is well known to be dense in $(1/2,1)$. Therefore $\cup_{n=1}^{\infty}H_{1}(n)$ and $\cup_{n=1}^{\infty}H_{2}(n)$ are both non-empty and property $1$ from Theorem \ref{Main thm} is satisfied. To show that property $2$ holds for this family fix $u\in H_{1}(n_0), v\in H_{2}(m_0)$ and $\epsilon>0$. Let $L_0=\max\{n_0,m_0\}$. Since the set of zeros of non-zero polynomials with coefficients belonging to the set $\{-1,0,1\}$ is dense, we can choose $u_1\in \cup_{n=L_0+1}^{\infty}H_{1}^{*}(n)$ such that $\|u-u_1\|_{\infty}<\epsilon$ and $u_1$ is not a zero for any non-zero polynomial of degree at most $L_0$ with coefficients belonging to the set $\{-2,-1,0,1,2\}$. Here we used the fact that the set of zeros of non-zero polynomials of degree at most $L_0$ with coefficients belonging to the set $\{-2,-1,0,1,2\}$ is a finite set. We've shown that property $2a$ from Theorem \ref{Main thm} is satisfied by this choice of $u_1$. It remains to show that property $2b.$ is also satisfied. As such let us fix $\epsilon'>0$. For an arbitrary choice of $v^*\in V$ one can check that if $(u_1,v^*)\in H_{3}(L_0),$ then there exists $(\kappa_{i,0}),\ldots,(\kappa_{i,L_0})\in \{-2,-1,0,1,2\}^{L_0+1}$ such that one of these sequences is non-zero and 
\begin{equation}
\label{Polynomial}
\sum_{i=0}^{L_0}\kappa_{i,0}(u_1)^i+\left(\sum_{i=0}^{L_0}\kappa_{i,1}(u_1)^i\right)v^*+\cdots \left(\sum_{i=0}^{L_0}\kappa_{i,L_0}(u_1)^i\right)(v^*)^{L_0}=0.
\end{equation} 

Equation \eqref{Polynomial} shows that if $(u_1,v^*)\in H_{3}(L_0)$ then $v^{*}$ is the zero of a polynomial of degree $L_0$ with coefficients belonging to the set $\{\sum_{i=0}^{L_0}\kappa_i(u_1)^i:(\kappa_i)\in\{-2,-1,0,1,2\}^{L_0+1}\}.$ Moreover, since one of $(\kappa_{i,0}),\ldots,(\kappa_{i,L_0})$ is non-zero and $u_1$ is not the zero of any polynomial of degree at most $L_0$ with coefficients belonging to the set $\{-2,-1,0,1,2\},$ we may conclude that if $(u_1,v^{*})\in H_{3}(L_0)$ then $v^{*}$ is the root of a non-zero polynomial with coefficients belonging to the set  $\{\sum_{i=0}^{L_0}\kappa_i(u_1)^i:(\kappa_i)\in\{-2,-1,0,1,2\}^{L_0+1}\}$ of degree at most $L_0$. There are only finitely many such $v^*$. Therefore appealing to the density of $\cup_{n=1}^{\infty}H_{2}(n)$ we may pick $v_1\in \cup_{n=L_0+1}^{\infty}H_{2}^{*}(n)$ such that $\|v-v_1\|_{\infty}<\epsilon'$ and $(u_1,v_1)\notin H_{3}(L_0)$. Therefore property $2b$ is satisfied for this choice of $v_1.$ By Theorem \ref{Main thm} there exists $\Phi_{\u^*,\v^*}$ within this family with no exact overlaps such that $(\Delta_n(\Phi_{\u^*,\v^*}))$ converges to zero arbitrarily fast.
\end{example}


\begin{thebibliography}{100}
\bibitem{Bak2} S. Baker, \textit{Iterated function systems with super-exponentially close cylinders}, arXiv:1909.04343 [math.DS].
\bibitem{Bak} S. Baker, \textit{Overlapping iterated function systems from the perspective of Metric Number Theory,} 	arXiv:1901.07875 [math.DS].
\bibitem{BarKae} B. B\'{a}r\'{a}ny, A. K\"{a}enm\"{a}ki, \textit{Superexponential condensation without exact overlaps,} arXiv:1910.04623 [math.DS].
\bibitem{Chen} C. Chen, \textit{Self-similar sets with super-exponential close cylinders,} 	arXiv:2004.14037 [math.CA].
\bibitem{FengHu} D-J. Feng, H. Hu, \textit{Dimension theory of iterated function systems. }
Comm. Pure Appl. Math. 62 (2009), no. 11, 1435-–1500. 
\bibitem{Hochman2} M. Hochman, \textit{Dimension theory of self-similar sets and measures,} In B. Sirakov, P. N. de Souza, and M. Viana, editors, Proceedings of the International Congress of Mathematicians (ICM 2018), pages 1943-–1966. World
Scientific, 2019.
\bibitem{Hochman} M. Hochman, \textit{On self-similar sets with overlaps and inverse theorems for entropy,} Ann. Math., {\bf 180}, (2014), 773--822.
\bibitem{Hochman3} M. Hochman, \textit{On self-similar sets with overlaps and inverse theorems for entropy in $\R^d$}, Mem. Amer. Math. Soc (to appear). 
\bibitem{Hut}  J. Hutchinson, \textit{Fractals and self-similarity}, Indiana Univ. Math. J. 30 (1981), no. 5, 713–-747.
\bibitem{Rap} A. Rapaport, \textit{Proof of the exact overlaps conjecture for systems with algebraic contractions,} arXiv:2001.01332 [math.DS].
\bibitem{Shm} P. Shmerkin, \textit{$L^q$ dimensions of self-similar measures, and applications: a survey,} 	arXiv:1907.07121 [math.DS].
\bibitem{Shm2} P. Shmerkin, \textit{On Furstenberg's intersection conjecture, self-similar measures, and the $L^q$ norms of convolutions,} Ann. of Math. (2) 189 (2019), no. 2, 319–-391. 
\bibitem{Varju} P. Varju, \textit{On the dimension of Bernoulli convolutions for all transcendental parameters,} Ann. of Math. (2) 189 (2019), no. 3, 1001-–1011.
\end{thebibliography}
\end{document}